\documentclass[11pt]{amsart}
\usepackage{tgpagella}
\usepackage{euler}
\usepackage[T1]{fontenc}
\usepackage{amsmath, amssymb}
\usepackage[hidelinks]{hyperref}
\usepackage[english]{babel}
\usepackage{mathrsfs}
\usepackage{eucal}
\usepackage[all]{xy}
\newtheorem{thm}{Theorem}[section]
\newtheorem*{theorem}{Theorem}
\newtheorem{defn}[thm]{Definition}
\newtheorem{dig}[thm]{Digression}
\newtheorem{lemma}[thm]{Lemma}

\newtheorem{cor}[thm]{Corollary}
\newtheorem{rmk}[thm]{Remark}

\newtheorem{conj}[thm]{Conjecture}
\newtheorem{fact}[thm]{Fact}
\newtheorem{question}[thm]{Question}
\newcommand\e\epsilon

\def \u{\mathcal U}

\newcommand{\cstar}{$\mathrm{C}^*$}

\makeatletter

\def\indsym#1#2{%
  \setbox0=\hbox{$\m@th#1x$}%
  \kern\wd0%
  \hbox to 0pt{\hss$\m@th#1\mid$\hbox to 0pt{$\m@th#1^{#2}$}\hss}%
  \lower.9\ht0\hbox to 0pt{\hss$\m@th#1\smile$\hss}%
  \kern\wd0}

\def\nindsym#1#2{%
  \setbox0=\hbox{$\m@th#1x$}%
  \kern\wd0%
  \hbox to 0pt{\hss$\m@th#1\not$\kern1.4\wd0\hss}
  \hbox to 0pt{\hss$\m@th#1\mid$\hbox to 0pt{$\m@th#1^{\,#2}$}\hss}%
  \lower.9\ht0\hbox to 0pt{\hss$\m@th#1\smile$\hss}%
  \kern\wd0}

\def\dotminussym#1#2{%
  \setbox0=\hbox{$\m@th#1-$}%
  \kern.5\wd0%
  \hbox to 0pt{\hss\hbox{$\m@th#1-$}\hss}%
  \raise.6\ht0\hbox to 0pt{\hss$\m@th#1.$\hss}%
  \kern.5\wd0}

\def \R{\mathcal R}

\title{On Popa's Factorial Commutant Embedding Problem}
\author{Isaac Goldbring}
\thanks{I. Goldbring was partially supported by NSF CAREER grant DMS-1349399.}

\address{Department of Mathematics\\University of California, Irvine, 340 Rowland Hall (Bldg.\# 400),
Irvine, CA 92697-3875}
\email{isaac@math.uci.edu}
\urladdr{http://www.math.uci.edu/~isaac}

\begin{document}

\begin{abstract}
An open question of Sorin Popa asks whether or not every $\R^\u$-embeddable factor admits an embedding into $\R^\u$ with factorial relative commutant.  We show that there is a locally universal McDuff II$_1$ factor $M$ such that every property (T) factor admits an embedding into $M^\u$ with factorial relative commutant.  We also discuss how our strategy could be used to settle Popa's question for property (T) factors if a certain open question in the model theory of operator algebras has a positive solution.
\end{abstract}

\maketitle

\section{Introduction}

In this note, all II$_1$ factors are assumed to be separable unless they are ultrapowers.  $\R$ denotes the hyperfinite II$_1$ factor.  $\u$ denotes an arbitrary nonprincipal ultrafilter on $\mathbb N$.  We say that a factor is \textbf{embeddable} if it embeds into $\R^\u$.  In order to avoid any set-theoretic subtleties, we also assume that the Continuum Hypothesis (CH) holds.\footnote{It would be interesting to investigate if any of our results depend on set theory}.

The starting point of this note is the following question of Popa:

\begin{question}[The factorial commutant embedding problem (FCEP)]
Suppose that $N$ is an embeddable factor.  Is there an embedding $\pi:N\hookrightarrow \R^\u$ such that $\pi(N)'\cap \R^\u$ is a factor?
\end{question}

The question is known to have a positive answer in some cases, e.g. $N=\R$ \cite[Proposition 12]{DL} and $N=\operatorname{SL}_3(\mathbb Z)$ \cite[Section 1.7]{popa}, but seems to be wide-open in general.  The question itself even seems to be open for the class of property (T) factors.

The main result of this note, proven in Section 2, is that there is a McDuff II$_1$ factor making the conclusion of the FCEP true for all property (T) factors:

\begin{theorem}
There is a locally universal McDuff II$_1$ factor $M$ such that, for any property (T) factor $N$, there is an embedding $\pi:N\hookrightarrow M^\u$ such that $\pi(N)'\cap M^\u$ is a factor.
\end{theorem}

We recall that a \textbf{locally universal} factor is one whose ultrapower contains all (separable) II$_1$ factors.  Locally universal factors were first shown to exist in \cite[Example 6.4(2)]{modeloperator3}, thus providing a ``poor man's resolution'' to the Connes Embedding Problem (CEP).  Thus, in some sense, our theorem is a ``poor man's resolution'' to the FCEP for property (T) factors.

Recently, a negative solution to the CEP was announced in \cite{quantum}; if correct, it would imply that a locally universal factor is not embeddable.  It thus makes sense to wonder whether or not $M$ as in the previous theorem can be taken to be embeddable if one restricts attention to embeddable property (T) factors; we discuss a hurdle to this being true in Section 3, where we also discuss how the success of this approach to settle the FCEP for property (T) factors is connected to an open question about so-called infinitely generic embeddable factors.

Popa's question was given a geometric reformulation by Nate Brown in \cite[Proposition 5.2]{Brown}, who showed that an embedding $\pi:N\hookrightarrow \R^\u$ has factorial commutant if and only if $[\pi]$ is an extreme point in the convex-like space $\operatorname{Hom}(N,\R^\u)$ of embeddings of $N$ into $\R^\u$ modulo unitary equivalence.  Scott Atkinson \cite[Theorem 5.4]{Scott} showed a similar result when $\R$ is replaced by a McDuff factor.  Consequently, our result shows that, for the $M$ as in the above theorem, $\operatorname{Hom}(N,M^\u)$ has an extreme point for any property (T) factor $N$.

Our proofs use ideas from model theory although we do our best to provide logic-free definitions of the main concepts.  In fact, the proof of the main theorem is mainly obtained by combining results from our earlier works \cite{ecfactor} and \cite{Gold2}.

We would like to thank Scott Atkinson and Srivatsav Kunnawalkam Elayavalli for bringing Popa's question to our attention and for useful conversations regarding this work.  We would also like to thank Sorin Popa for providing historical context for his question and for providing us with some references.

\section{The main theorem}

We recall the following definition:

\begin{defn}
Suppose that $M$ is a II$_1$ factor with a subfactor $N$.  We say that $N$ has \textbf{w-spectral gap in $M$} if $N'\cap M^\u=(N'\cap M)^\u$.
\end{defn}

We remind the reader that property (T) factors have w-spectral gap in any extension.  

We will need the following notion from model theory, defined in ultrapower\footnote{Read:  operator algebraist-friendly} terms.

\begin{defn}
If $M$ is a subfactor of the factor $Q$, we say that $M$ is \textbf{existentially closed in $Q$} if there is an embedding $Q\hookrightarrow M^\u$ that restricts to the diagonal embedding $M\hookrightarrow M^\u$.  We say that the II$_1$ factor $M$ is \textbf{existentially closed (e.c.)} if it is existentially closed in all extensions.
\end{defn}

The following is \cite[Section 2]{nomodcomp}.
\begin{fact}
An e.c. factor is locally universal and McDuff.
\end{fact}

\begin{defn}
A class $\mathcal C$ of II$_1$ factors is said to be \textbf{extensive}\footnote{In the model-theoretic literature, one would say that $\mathcal C$ is model-consistent with the class of II$_1$ factors.  We prefer the above terminology.} if every II$_1$ factor embeds in an element of $\mathcal C$.
\end{defn}

The following is well-known (see, e.g. \cite[Fact 2.8]{usvy}).

\begin{fact}
The class of e.c. factors is extensive.
\end{fact}

%Abstract things.  Extensive class of factors.  An extensive model-complete class of factors has the property that any two elements are e.e.  The class of infinitely generic factors is extensive and model complete.  Also any extensive model-complete class is such that all of its elements are e.c.

\begin{fact}
E.c. factors are locally universal and McDuff.
\end{fact}

The following appears in \cite{Gold2}:
\begin{fact}\label{bicomm}
Suppose that $N$ is a w-spectral gap subfactor of the e.c. factor $M$.  Then $(N'\cap M)'\cap M=N$.
\end{fact}

\begin{thm}
Suppose that $N$ is a w-spectral gap subfactor of the e.c. factor $M$.  Then $N'\cap M$ is a factor.
\end{thm}

\begin{proof}
Set $P:=N'\cap M$.  We show that $P$ is a factor.  Take $x\in P$ such that $[x,y]=0$ for all $y\in P$.  Then $x\in (N'\cap M)'\cap M=N$, so $x\in N$.  Now suppose that $z\in N$.  Then since $x\in P$, we have $[x,z]=0$.  So $x\in Z(N)=\mathbb C$, as desired.
\end{proof}

\begin{cor}\label{almost}
Suppose that $N$ is a w-spectral gap subfactor of the e.c. factor $M$.  Then $N'\cap M^\u$ is a factor.
\end{cor}

\begin{proof}
$N'\cap M^\u=(N'\cap M)^\u$ and the ultrapower of a factor is once again a factor.
\end{proof}

Although we won't need the following result, it might be of independent interest:

\begin{dig}
If $N$ is a w-spectral gap of the e.c. factor $M$, then $N'\cap M$ is a locally universal McDuff II$_1$ factor.
\end{dig}

\begin{proof}
Once again, set $P:=N'\cap M$.  We first show that $P$ is locally universal.  Let $Q$ be any II$_1$ factor.  Since $M$ is e.c. in $M\otimes Q$, we have an embedding $M\otimes Q\hookrightarrow M^\u$ that restircts to the diagonal embedding $M\hookrightarrow M^\u$.  In particular, $Q\hookrightarrow M'\cap M^\u\subseteq N'\cap M^\u=(N'\cap M)^\u=P^\u$, whence $P$ is locally universal.

%Suppose that $\varphi(x)<\epsilon$ is a satisfiable existential condition; it suffices to show that it can be satisfied in $P$.  Towards that end, suppose that $p<\epsilon$ is satisfied in $Q$.  Take $F\subseteq N$ finite and $\delta>0$ witnessing w-spectral gap of $M$ for $\eta>0$ small enough.  Then $M\otimes Q\models \inf_x\max(p(x),\max_{y\in F}\|[x,y]\|_2)<\delta$, whence so does $M$.  Let $a$ realize the existential in $M$.  Then there is $b\in P$ such that $d(a,b)<\eta$.  If $\eta$ is small enough, then $p(b)<\epsilon$, as desired.

Since $P$ is locally universal, it follows that $P$ is a II$_1$ factor.

Finally, we show that $P$ is McDuff.  It suffices to show that $M_2(\mathbb C)$ embeds in $P'\cap P^\u$.  Take an embedding $M\otimes M_2(\mathbb C)\hookrightarrow M^\u$ restricting to the diagonal embedding of $M\hookrightarrow M^\u$.  As in the previous argument, this embedding sends $M_2(\mathbb C)$ into $P^\u$.  Moreover, since $M'\cap M^\u\subseteq P'\cap M^\u$, this embedding sends $M_2(\mathbb C)$ into $P'\cap P^\u$, as desired.
%
% It suffices to show that, given any finite $F\subseteq P$, that there is a copy of $M_2$ inside of $P$ that almost commutes with $F$.  Since $M\otimes M_2$ thinks that there is a copy of $M_2$ that commutes with $G\subseteq N$ and $F$, it follows that $M$ thinks that there are almost matrix units that almost commute with $G$ and almost commute with $F$.  By spectral gap, these almost matrix units are near almost matrix units that actually commute with $N$, so lie in $P$.  By stability, these almost matrix units in $P$ are near actual matrix units in $P$ which still almost commute with $F$. 
\end{proof}

Returning to the main thread, at this moment, we simply have that every property (T) factor $N$ embeds in a II$_1$ factor $M$ such that the diagonal embedding $N\hookrightarrow M^\u$ has factorial relative commutant.  We would like a single $M$ that works for all property (T) factors.  This leads us to the following:

\begin{defn}
II$_1$ factors $M_1$ and $M_2$ are said to be \textbf{elementarily equivalent}, denoted $M_1\equiv M_2$, if $M_1^\u\cong M_2^\u$.\footnote{This is an evil, logic-free, definition, and makes heavy use of our standing CH assumption.}
\end{defn}

The following observation is obvious but crucial:

\begin{lemma}\label{obvious}
If $M_1\equiv M_2$ and $N$ is a II$_1$ factor, then $N$ admits an embedding into $M_1^\u$ with factorial relative commutant if and only if $N$ admits an embedding into $M_2^\u$ with factorial relative commutant.
\end{lemma}

Consequently, \textbf{if} all e.c. factors were elementarily equivalent, then our main theorem would follow from Corollary \ref{almost} by taking any e.c. fcactor.

However, while still an open problem, it is highly unlikely (at least in this author's opinion) that all e.c. factors are elementarily equivalent.  Instead, we look to an important subclass of these factors for which all members are elementarily equivalent.  First, we need:

\begin{defn}
Suppose that $M_1$ is a subfactor of the II$_1$ factor $M_2$.  We say that $M_1$ is an \textbf{elementary subfactor} of $M_2$, denoted $M_1\preceq M_2$, if there is an isomorphism $M_1^\u\cong M_2^\u$ that fixes the diagonal images of $M_1$.\footnote{Again, an evil logic-free definition taking full advantage of our standing CH assumption.}
\end{defn}

\begin{defn}
If $\mathcal C$ is a class of II$_1$ factors, we say that $\mathcal C$ is \textbf{model-complete} if, whenever $M_1$ and $M_2$ are elements of $\mathcal C$ with $M_1\subseteq M_2$, then $M_1\preceq M_2$. 
\end{defn}

The following facts follow easily from the definitions:

\begin{lemma}\label{extensivemc}
Suppose that $\mathcal C$ is an extensive, model-complete class of II$_1$ factors.  Then:
\begin{enumerate}
\item Every element of $\mathcal C$ is an e.c. factor.
\item If $M_1$ and $M_2$ belong to $\mathcal C$, then $M_1\equiv M_2$.
\end{enumerate}
\end{lemma}

The following is a combination of \cite[Proposition 5.7, Proposition 5.10, and Proposition 5.14]{ecfactor}:

\begin{fact}
There is an extensive, model-complete class of II$_1$ factors.  In fact, there is a maximum such class $\mathcal G$.
\end{fact}

\begin{defn}
Elements of $\mathcal G$ are called \textbf{infinitely generic} II$_1$ factors.
\end{defn}

We now have the main result:

\begin{thm}
Suppose that $M$ is an infinitely generic II$_1$ factor.  Then if $N$ is a property (T) factor, then $N$ admits an embedding into $M^\u$ with factorial relative commutant.
\end{thm}

\begin{proof}
Take an infinitely generic II$_1$ factor $M_1$ with $N\subseteq M_1$.  By Lemma \ref{extensivemc}(1), $M_1$ is e.c. whence $N'\cap M^\u$ is a factor by Corollary \ref{almost}.  By Lemma \ref{extensivemc}(2), $M\equiv M_1$, whence we are done by Lemma \ref{obvious}.
\end{proof}

\begin{rmk}
All that we used about property (T) factors is that they automatically have w-spectral gap in any extension.  Are there other II$_1$ factors with this property?
\end{rmk}

\section{The case of embeddable factors}

We now consider what happens when we restrict to embeddable factors.  All notions from the last section relativize to this setting.  For example, by an e.c. embeddable factor we mean an embeddable factor that is e.c. in all embeddable extensions.  Similarly, one can define the class of infinitely generic embeddable factors, which forms a subclass of the class of e.c. embeddable factors.  The class of e.c. embeddable factors and the subclass of infinitely generic embeddable factors are both extensive in the class of embeddable factors.  See \cite{ecfactor} for more details on this.

\begin{conj}\label{embedconj}
Suppose that $N$ is a w-spectral gap subfactor of the e.c. embeddable factor $M$.  Then $(N'\cap M)'\cap M=N$. 
\end{conj}

Why is it not the case that Conjecture \ref{embedconj} is simply a theorem?  Well, the proof of Fact \ref{bicomm} uses the fact that if $N$ is a w-spectral gap subfactor of the e.c. factor $M$, then $M$ is e.c. in the amalgamated free product $M*_N (N\otimes L(\mathbb Z))$.  If $M$ is an e.c. embeddable factor, then we could only conclude that $M$ is e.c. in $M*_N (N\otimes L(\mathbb Z))$ \textit{if} we knew that $M*_N (N\otimes L(\mathbb Z))$ is also embeddable.  However, it is unknown at the moment whether or not this is the case.  

\begin{question}\label{amal}
Does taking amalgamated free products of embeddable factors with property (T) base preserve embeddability?
\end{question}

Thus, we just argued that a positive answer to Question \ref{amal} yields a positive solution to Conjecture \ref{embedconj}.

Suppose we have a positive solution to Conjecture \ref{embedconj}.  Since $\R$ is an e.c. embeddable factor (see \cite[Lemma 2.1]{ecfactor}), once again, \textbf{if} all e.c. embeddable factors were elementarily equivalent, we would actually arrive at a positive solution to the FCEP for property (T) factors.  Once again, we believe this to be highly doubtful.  Passing to infinitely generic embeddable factors  and noting that the rest of the arguments of the previous section go through, we get:

\begin{thm}
Suppose that Conjecture \ref{embedconj} has a positive answer and that $M$ is an infinitely generic embeddable factor.  Then every embeddable property (T) factor admits a factorial embedding into $M^\u$.
\end{thm}

In light of the previous theorem and recalling Popa's original question, we arrive at the obvious question:

\begin{question}\label{infgen}
Is $\R$ an infinitely generic embeddable factor?
\end{question}

\begin{cor}
If Conjecture \ref{embedconj} is true and Question \ref{infgen} has a positive answer, then the FCEP for property (T) factors has a positive solution.
\end{cor}

In \cite[Proposition 5.21]{ecfactor}, it was claimed that $\R$ is an infinitely generic embeddable factor.  However, the proof there is horribly flawed and the question is still open at this time.  Let us point out:

\begin{lemma}\label{infgenequiv}
The following statements are equivalent:
\begin{enumerate}
\item $\R$ is infinitely generic embeddable factor.
\item There is an infinitely generic embeddable factor $M$ such that $\R\equiv M$.
\end{enumerate}
\end{lemma}

\begin{proof}
To prove the nontrivial direction, suppose that $M$ is an infinitely generic embeddable factor such that $\R\equiv M$.  Fixing an embedding $\R\hookrightarrow M$, we have that this embedding is automatically elementary.\footnote{This is well-known to those working in the model theory of operator algebras, but we sketch a quick proof for the sake of the reader.  Fix an isomorphism $M^\u\cong \R^\u$ and note that the induced embedding $\R\hookrightarrow \R^\u$ is conjungate to the diagonal embedding by the easy direction of the main result of \cite{jung}.  Consequently, there is an isomorphism $M^\u\cong \R^\u$ fixing the diagonal image of $\R$.}  We now quote \cite[Proposition 5.17]{ecfactor}, which implies that an elementary subfactor of an infinitely generic embeddable factor is an infinitely generic embeddable factor.
\end{proof}

There is another class of e.c. (embeddable) factors with the property that any two members are elementarily equivalent, namely the so-called \textbf{finitely generic (embeddable) factors} (see \cite[Section 6]{ecfactor} or \cite[Section 3]{Gold1} for a precise definition).  This class is also model-complete; in fact, by \cite[Corollary 3.12]{Gold1}, if a factor is e.c. in a finitely generic (embeddable) factor, then it is also a finitely generic (embeddable) factor.  Consequently, $\R$ is a finitely generic embeddable factor.\footnote{With some revisionist history, one can use this fact to give an alternate definition of finitely generic embeddable factors, namely an embeddable factor $M$ is a finitely generic embeddable factor if and only if:  (i)  $M\equiv \R$, and (ii) whenever $M\subseteq N$ and $N\equiv \R$, then $M$ is an elementary subfactor of $N$.}  Thus, at first glance, it might seem promising to look at this class instead.  Unfortunately, this class is far from extensive:

\begin{fact}(\cite{AGK})
$\R$ is the unique finitely generic embeddable factor.
\end{fact}

\begin{rmk}
In the case of groups, the finitely generic and the infinitely generic groups are different (see \cite[Theorem 11]{Mac}).  Perhaps similar proofs could be used to negatively answer Question \ref{infgen}.
\end{rmk}


\begin{thebibliography}{99}

\bibitem{Scott} S. Atkinson, \textit{Convex sets associated to \cstar-algebras}. J. Funct. Anal. \textbf{271} (2016), 1604-1651.

\bibitem{AGK} S. Atkinson, I. Goldbring, and S. K. Elayavalli, \textit{On II$_1$ factors with the generalized Jung property}, manuscript in preparation.

\bibitem{mtfms} I. Ben Yaacov, A. Berenstein, C. W. Henson, and A. Usvyatsov, \textit{Model theory for metric structures}, Model theory with applications to algebra and analysis. 2, 315- 427, London Math. Soc. Lecture Note Ser. (350), Cambridge Univ. Press, Cambridge, 2008.

\bibitem{Brown} N. P. Brown, \textit{Topological dynamical systems associated to II$_1$ factors}, Adv. Math. \textbf{227}, 1665-1699, 2011. With an appendix by Narutaka Ozawa.

\bibitem{DL} J. Dixmier and E. C. Lance, \textit{Deux nouveaux facteurs de type II$_1$}, Invent. Math. \textbf{7} (1969), 226-234.

\bibitem{ecfactor} I. Farah, I. Goldbring, B. Hart, and D. Sherman, \textit{Existentially closed II$_1$ factors}, Fundamenta Mathematicae \textbf{233} (2016), 173-196. 

\bibitem{modeloperator3} I. Farah, B. Hart, and D. Sherman, \textit{Model theory of operator algebras III: elementary equivalence and II$_1$ factors}, Bull. Lond. Math. Soc. \textbf{46} (2014), 609-628.

\bibitem{Gold1} I. Goldbring, \textit{Enforceable operator algebras}, to appear in the Journal of the Institute of Mathematics of Jussieu.

\bibitem{Gold2} I. Goldbring, \textit{Spectral gap and definability}, to appear in Beyond First-order Model Theory Volume 2.  arXiv 1805.02752.

\bibitem{nomodcomp} I. Goldbring, B. Hart, and T. Sinclair, \textit{The theory of tracial von Neumann algebras does not have a model companion}, Journal of Symbolic Logic \textbf{78} (2013), 1000-1004.

\bibitem{jung} K. Jung, \textit{Amenability, tubularity, and embeddings into $R^\omega$}, Math. Ann \textbf{338} (2007), 241-248.

\bibitem{Mac} A. Macintyre, \textit{On algebraically closed groups}, Annals of Mathematics \textbf{96} (1972), 53-97.

\bibitem{quantum} Z. Ji, A. Natarajan, T. Vidick, J. Wright, and H. Yuen, \textit{MIP$^*=$RE}, arxiv 2001.04383.

\bibitem{popa} S. Popa, {\it Independence properties in subalgebras of ultraproduct} II$_1$ {\it factors}, Journal of Functional Analysis 
{\bf 266} (2014), 5818-5846.

\bibitem{usvy} A. Usvyatsov, \textit{Generic separable metric structures}, Topology and its Applications \textbf{155} (2008), 1607-1617.


\end{thebibliography}
\end{document}